\documentclass[final]{elsarticle}
\usepackage[top=2.5cm, bottom=2.5cm, left=3cm, right=3cm]{geometry}

\makeatletter
\def\ps@pprintTitle{%
 \let\@oddhead\@empty
 \let\@evenhead\@empty
 \def\@oddfoot{\centerline{\thepage}}%
 \let\@evenfoot\@oddfoot}
\makeatother

\usepackage{hyperref}
\usepackage{graphicx}
\usepackage{amsfonts}
\usepackage{amssymb}
\usepackage{amsmath}
\usepackage{amsthm}
\usepackage{float}
\usepackage{exscale}
\usepackage{relsize}
\usepackage{bm}
\usepackage{subeqnarray}
\usepackage{fancyhdr}
\usepackage{lastpage}
\usepackage{layout}
\usepackage{cleveref}
\usepackage{appendix}
 
\newtheorem{thm}{Theorem}[section]
\newtheorem{corollary}[thm]{Corollary}
\newtheorem{proposition}[thm]{Proposition}
\newtheorem{lemma}[thm]{Lemma}
\newtheorem{definition}[thm]{Definition}

\newcommand{\trace}{\mathrm{Tr}}

\begin{document}

\begin{frontmatter}

\title{\textbf{Improvement on a Generalized Lieb's Concavity Theorem}}
\author{De Huang\fnref{myfootnote} }
\address{Applied and Computational Mathematics, California Institute of Technology, Pasadena, CA 91125, USA}
\fntext[myfootnote]{E-mail address:\ dhuang@caltech.edu.}
 
\begin{abstract}
We show that Lieb's concavity theorem holds more generally for any unitary invariant matrix function $\phi:\mathbf{H}_+^n\rightarrow \mathbb{R}_+^n$ that is concave and satisfies H\"older's inequality. Concretely, we prove the joint concavity of the function $(A,B) \mapsto\phi\big[(B^\frac{qs}{2}K^*A^{ps}KB^\frac{qs}{2})^{\frac{1}{s}}\big] $ on $\mathbf{H}_+^n\times\mathbf{H}_+^m$, for any $K\in \mathbb{C}^{n\times m}$ and any $s,p,q\in(0,1], p+q\leq 1$. This result improves a recent work by Huang for a more specific class of $\phi$.
\end{abstract}

\begin{keyword}
Lieb's concavity theorem, matrix functions, symmetric forms, operator interpolation, majorization. 
\vspace{2mm}
\MSC[2010] 47A57, 47A63, 15A42, 15A16
\end{keyword}

\end{frontmatter}

\section{Introduction}
Lieb's Concavity Theorem \cite{LIEB1973267}, as one of the most celebrated results in the study of trace inequalities, states that the function 
\begin{equation}\label{eqt:LCT}
(A,B)\ \longmapsto\ \trace[K^*A^pKB^q]
\end{equation}
is jointly concave on $\mathbf{H}_+^n\times\mathbf{H}_+^m$, for any $K\in \mathbb{C}^{n\times m}$, $p,q\in(0,1], p+q\leq 1$. Here $\mathbf{H}_+^n$ is the convex cone of all $n \times n$ Hermitian, positive semidefinite matrices. Recently, Huang \cite{huang2019generalizing} generalized Lieb's result to the concavity of 
\begin{equation}\label{eqt:ktraceLCT}
(A,B) \ \longmapsto\ \trace_k\big[(B^\frac{qs}{2}K^*A^{ps}KB^\frac{qs}{2})^{\frac{1}{s}}\big]^\frac{1}{k}, 
\end{equation}
where the $k$-trace $\trace_k(A)$ of a matrix $A\in \mathbb{C}^{n\times n}$ is defined as 
\[\trace_k(A) = \sum_{1\leq i_1<i_2<\cdots<i_k\leq n} \lambda_{i_1}\lambda_{i_2}\cdots \lambda_{i_k},\quad 1\leq k\leq n,\]
with $\lambda_1,\lambda_2,\dots,\lambda_n$ being the eigenvalues of $A$, counting multiplicities. In this paper, we further improve Huang's results from $k$-trace to any unitary invariant matrix function $\phi:\mathbf{H}_+^n\rightarrow \mathbb{R}_+^n$ that is concave and satisfies H\"older's inequality, i.e. $\phi(|AB|)\leq \phi(|A|^p)^\frac{1}{p}\phi(|B|^q)^\frac{1}{q},p,q\in[1,+\infty],\frac{1}{p}+\frac{1}{q}=1$. For such a matrix function $\phi$, we will show that 
\begin{equation}\label{eqt:PhiLCT}
(A,B) \ \longmapsto\ \phi\big[(B^\frac{qs}{2}K^*A^{ps}KB^\frac{qs}{2})^{\frac{1}{s}}\big] 
\end{equation}
is jointly concave on $\mathbf{H}_+^n\times\mathbf{H}_+^m$, for any $K\in \mathbb{C}^{n\times m}$ and any $s,p,q\in(0,1], p+q\leq 1$. Note that the trace and the general $k$-traces $\trace_k[\cdot]^\frac{1}{k}$ are unitary invariant, concave and satisfy H\"older's inequality. 

Huang's proof for the concavity of \eqref{eqt:ktraceLCT} was based on an operator interpolation theory by Stein \cite{stein1956interpolation}. He obtained an interpolation inequality on $k$-trace from Stein's result by interpreting $\trace_k[A] = \trace[\wedge^k A]$, where $\wedge^k$ stands for the $k$-fold antisymmetric tensor product. However, the use of interpolation actually only requires the unitary invariance and the H\"older property of $k$-trace, rather than its specific form. We hence consider to derive similar interpolation inequalities on more general symmetric functions satisfying H\"older's inequality, by adopting majorization techniques in Huang's framework. Our approach was inspired by a recent work of Hiai et al. \cite{hiai2017generalized}, where they used majorization theories to obtain generalized log-majorization theorems, with application to a strengthened version of the multivariate extension of the Golden-Thompson inequality by Sutter et el. \cite{sutter2017multivariate}.

\subsection*{outline}
The rest of the paper is organized as follows. \Cref{sec:Notations&MainResults} is devoted to introductions of general notations, the notion of symmetric forms and our main results. We will briefly review in \Cref{sec:Preparations} the theories of antisymmetric tensor, majorization and operator interpolation, and use these tools to prove some lemmas on symmetric forms. The proofs of our main theorems are presented in \Cref{sec:Proofs}. In \Cref{sec:Discussions} we discuss some potential improvements of our current results.

\section{Notations and Main Results}
\label{sec:Notations&MainResults}

\subsection{General conventions}
For any positive integers $n,m$, we write $\mathbb{C}^n$ for the $n$-dimensional complex vector spaces equipped with the standard $l_2$ inner products, and $\mathbb{C}^{n\times m}$ for the space of all complex matrices of size $n\times m$. Let $\mathbb{R}^n,\mathbb{R}_+^n,\mathbb{R}_{++}^n$ be $(-\infty,+\infty)^n,[0,+\infty)^n,(0,+\infty)^n$ respectively. Let $\mathbf{H}^n$ be the space of all $n\times n$ Hermitian matrices, $\mathbf{H}_+^n$ be the convex cone of all $n\times n$ Hermitian, positive semi-definite matrices, and $\mathbf{H}_{++}^n$ be the convex cone of all $n\times n$ Hermitian, positive definite matrices. We write $I_n$ for the identity matrix of size $n \times n$. Abusing notation, we will use $i$ sometimes as an integer index and sometimes as the imaginary unit $\sqrt{-1}$ without clarification if no confusion caused. We use $S_n$ to denote the symmetric group of all permutations of order $n$. 

For any $x=(x_1,\dots,x_n),y=(y_1,\dots,y_n)\in \mathbb{R}^n$, we write $x+y$ and $xy$ for the entry-wise sum and entry-wise product respectively, i.e.
\[x+y=(x_1+y_1,\dots,x_n+y_n),\quad xy = (x_1y_1,\dots,x_ny_n).\]
We say $x\leq y$ if $x_i\leq y_i,i=1,\dots,n$, and  $x< y$ if $x_i< y_i,i=1,\dots,n$. For any function scalar function $f:\mathbb{R}\rightarrow\mathbb{R}$, the extension of $f$ to a function from $\mathbb{R}^n$ to $\mathbb{R}^n$ is given by
\[f(x) = (f(x_1),\dots,f(x_n)),\quad x\in \mathbb{R}^n. \]

For any $A\in \mathbf{H}^n$, we use $\lambda_i(A)$ to denote the $i_{\text{th}}$ largest eigenvalue of $A$, i.e. $\lambda_1(A)\geq \lambda_2(A)\geq\cdots\geq \lambda_n(A)$, and write $\lambda(A) = (\lambda_1(A),\dots,\lambda_n(A))\in\mathbb{R}^n$. For any scalar function $f:\mathbb{R}\rightarrow\mathbb{R}$, the extension of $f$ to a function from $\mathbf{H}^n$ to $\mathbf{H}^n$ is given by 
\[f(A)=\sum_{i=1}^nf(\lambda_i(A))u_iu_i^*, \quad A\in \mathbf{H}^n,\]
where $u_1,u_2,\cdots,u_n\in\mathbb{C}^n$ are the corresponding normalized eigenvectors of $A$. A function $f$ is said to be operator monotone increasing (or decreasing) if $A\succeq B$ implies $f(A)\succeq f(B)$ (or $f(A)\preceq f(B)$); $f$ is said to be operator concave (or convex) on some set $S$, if
\[\tau f(A)+(1-\tau) f(B) \preceq f(\tau A+(1-\tau) B)\ (\text{or} \succeq f(\tau A+(1-\tau) B)),\] 
for any $A,B\in S$ and any $\tau\in[0,1]$. For example, the function $A\mapsto A^r$ is both operator monotone increasing and operator concave on $\mathbf{H}_+^n$ for $r\in[0,1]$ (the L\"owner-Heinz theorem \cite{lowner1934monotone}, \cite{heinz1951beitrage}, see also \cite{carlen2010trace}). One can find more discussions and analysis on matrix functions in \cite{carlen2010trace,Vershynina:2013}. For any $A\in\mathbb{C}^{n\times m}$, we write $|A|=(A^*A)^\frac{1}{2}$, and denote by $\|A\|_p$ the standard Schatten $p$-norm,
\begin{equation}
\|A\|_p = \trace[|A|^p]^\frac{1}{p}.
\end{equation}
In particular, we write $\|A\| = \|A\|_{\infty}=$ the largest singular value of $A$.

\subsection{Symmetric forms}
We start with continuous symmetric functions on $\mathbb{R}_+^n$ defined as follows.
\begin{definition}\label{def:SymmetricForm}
A continuous function $\phi:\mathbb{R}^n_+\rightarrow \mathbb{R}_+$ is a \textbf{symmetric form} if $\phi$ satisfies:
\begin{itemize}
\item Homogeneity: $\phi(t x) = t \phi(x)$, for any $x\in \mathbb{R}^n_+,t\in \mathbb{R}_+$.
\item Monotonicity: For any $x,y\in\mathbb{R}^n_+$, $\phi(x)\geq \phi(y)$ if $x\geq y$; $\phi(x)> \phi(y)$ if $x> y$.
\item Symmetry: $\phi(x) = \phi(Px)$ for any $x\in \mathbb{R}^n_+$ and any permutation $P\in S_n$.
\end{itemize}
A symmetric form $\phi$ is \textbf{H\"older}, if it satisfies H\"older's inequality,
\[\phi(xy) \leq \phi(x^p)^\frac{1}{p}\phi(y^q)^\frac{1}{q},\]
for any $x,y\in \mathbb{R}^n_+$ and any $p,q\in[1,+\infty],\frac{1}{p}+\frac{1}{q}=1$. A symmetric form $\phi$ is \textbf{concave (or convex)}, if 
\[\phi(\tau x+(1-\tau)y)\geq \tau \phi(x)+(1-\tau)\phi(y)\ (\text{or }\leq \tau \phi(x)+(1-\tau)\phi(y)),\]
for any $x,y\in \mathbb{R}^n_+$ and any $\tau\in[0,1]$.
\end{definition}

The domain of a symmetric form $\phi$ can be naturally extended from $\mathbb{R}_+^n$ to $\mathbf{H}_+^n$, by feeding $\phi$ the eigenvalues of a matrix in $\mathbf{H}_+^n$. 

\begin{definition}\label{def:MatrixForm}
The extension of a symmetric form $\phi$ to $\mathbf{H}_+^n$ is defined as
\[\phi(A) = \phi(\lambda(A)),\quad A\in \mathbf{H}_+^n,\]
where $\lambda(A) = (\lambda_1(A),\lambda_2(A),\dots,\lambda_n(A))$ are the eigenvalues of $A$. 
\end{definition}

The following properties of the matrix extension result directly from \Cref{def:SymmetricForm}. Moreover, we will see in \Cref{subsec:SymFormProperty} that the matrix extension of a symmetric form will inherit its concavity (or convexity) or H\"older property if it has any. 

\begin{proposition}
If $\phi$ is a symmetric form, then its extension to $\mathbf{H}_+^n$ satisfies: 
\begin{itemize}
\item Homogeneity: $\phi(t A) = t \phi(A)$, for any $A\in \mathbf{H}_+^n,t\in \mathbb{R}_+$.
\item Monotonicity: For any $A,B\in\mathbf{H}_+^n$, $\phi(A)\geq\phi(B)$ if $A\succeq B$; $\phi(A)>\phi(B)$ if $A\succ B$.
\item Unitary invariance: $\phi(U^*AU) = \phi(A)$ for any $A\in \mathbf{H}_+^n$ and any unitary matrix $U\in \mathbb{C}^{n\times n}$.
\end{itemize}
\end{proposition}

Generally, if a symmetric form $\phi$ is convex and, furthermore, positive definite, i.e. 
\[\phi(x)=0\Longleftrightarrow x = (0,0,\dots,0),\]
then $\phi$ is called a symmetric gauge function. A famous bijection theory of von Neumann \cite{von1937some} says that any unitary invariant matrix norm on $\mathbf{H}_+^n$ is the extension of some symmetric gauge function on $\mathbb{R}_+^n$. Note that a convex symmetric form $\phi$ is automatically H\"older. In this paper, our main results, however, are most related to symmetric forms that are \textbf{concave and H\"older}. Some examples of such class of symmetric forms are listed below.
\begin{enumerate}
\item The k-trace introduced in \cite{2018arXiv180805550H}:
\[\trace_k[x]^\frac{1}{k} = \left(\sum_{1\leq i_1<i_2<\cdots<i_k\leq n} x_{i_1}x_{i_2}\cdots x_{i_k}\right)^\frac{1}{k},\quad x\in\mathbb{R}_+^n,\quad 1\leq k\leq n.\]
\item The sum of rotated partial geometric means:
\[g_k(x) = \sum_{1\leq i_1<i_2<\cdots<i_k\leq n} (x_{i_1}x_{i_2}\cdots x_{i_k})^\frac{1}{k},\quad x\in\mathbb{R}_+^n,\quad 1\leq k\leq n.\]
\item The semi $p$-norm for $p\in (0,1]$:
\[\|x\|_p = \left(\sum_{i=1}^n x_i^p\right)^\frac{1}{p},\quad x\in \mathbb{R}_+^n.\]
\end{enumerate}

\subsection{Main Theorems}
Our main task is to generalize Lieb's concavity theorems from trace to symmetric forms that are concave and H\"older. Huang \cite{huang2019generalizing} applied operator interpolations to obtain generalizations of Lieb's concavity to k-traces $\phi(x) = \trace_k[x]^\frac{1}{k}$, which he used to derive concentration estimates on partial spectral sums of random matrices \cite{2018arXiv180805550H}. The interpolation part of his proof requires essentially the symmetry and H\"older property of $k$-trace. Recently, Hiai et al. \cite{hiai2017generalized}, combined theories of majorization and operator interpolation to extend the multivariate Golden-Thompson inequality to a more general form. Inspired by their work, we also adopt techniques of majorization to further extend Huang's results to the following.

\begin{lemma}\label{lem:GeneralEpstein}
Let $\phi$ be a symmetric form that is H\"older and concave. Then for any $s,r\in(0,1]$ and any $K\in \mathbb{C}^{n\times n}$, the function
\begin{equation}
A \ \longmapsto\ \phi\big((K^*A^{rs}K)^{\frac{1}{s}}\big) 
\label{eqt:function1}
\end{equation}
is concave on $\mathbf{H}_+^n$. 
\end{lemma}

\begin{thm}[Generalized Lieb's Concavity Theorem]\label{thm:GeneralLiebConcavity}
Let $\phi$ be a symmetric form that is H\"older and concave. Then for any $s,p,q\in(0,1],p+q\leq 1$, and any $K\in \mathbb{C}^{n\times m}$, the function 
\begin{equation}
(A,B) \ \longmapsto\ \phi\big((B^\frac{qs}{2}K^*A^{ps}KB^\frac{qs}{2})^{\frac{1}{s}}\big) 
\label{eqt:function2}
\end{equation}
is jointly concave on $\mathbf{H}_+^n\times\mathbf{H}_+^m$. 
\end{thm}

\begin{thm}\label{thm:GeneralLieb}
Let $\phi$ be a symmetric form that is H\"older and concave. Then for any $H\in \mathbf{H}^n$ and any $\{p_j\}_{j=1}^m\subset(0,1]$ such that $\sum_{j=1}^mp_j\leq1$, the function 
\begin{equation}
(A_1,A_2,\dots,A_m) \ \longmapsto\ \phi\big(\exp\big(H+\sum_{j=1}^mp_j\log A_j\big)\big)
\label{eqt:function3}
\end{equation}
is jointly concave on $(\mathbf{H}_{++}^n)^{\times m}$. In particular, $A\mapsto\phi\big(\exp(H+\log A)\big)$ is concave on $\mathbf{H}_{++}^n$.
\end{thm}

\Cref{lem:GeneralEpstein} is an extension of the concave part of Lemma 2.8 in \cite{Carlen2008} (see also \cite{epstein1973remarks}), which is a direct consequence of the original Lieb's concavity theorem. We will first apply the technique of operator interpolation to prove \Cref{lem:GeneralEpstein} independently, and then use it to derive the other results. \Cref{thm:GeneralLiebConcavity} is our generalized Lieb's concavity theorem, which not only extends the original Lieb's concavity to any symmetric form that is concave and H\"older, but also strengthens its form by adding the power $s$. \Cref{thm:GeneralLieb} is a generalization of Corollary 6.1 in \cite{LIEB1973267}. Lieb proved the original trace version by checking the non-positiveness of the second order directional derivatives (or Hessians). Huang \cite{2018arXiv180805550H} imitated Lieb's derivative arguments and proved the concavity of $A\mapsto\trace_k\big[\exp(H+\log A)\big]^\frac{1}{k}$, which he then generalized from $m=1$ to $m\geq 1$ in \cite{huang2019generalizing}. We here further extend this result to symmetric forms that are concave and H\"older. The proofs of our main results are diverted to \Cref{sec:Proofs}.

\section{Preparations}
\label{sec:Preparations}

\subsection{Antisymmetric tensors} 
\label{subsec:AntiTensor}
Theories of antisymmetric tensors have been useful tools for deriving important majorization relations between eigenvalues of matrices (see e.g. \cite{hiai2010matrix,Rotman:2114272}). For any $1\leq k\leq n$, let $\wedge^k(\mathbb{C}^n)$ denote the $k$-fold antisymmetric tensor space of $\mathbb{C}^n$, equipped with the inner product 
\begin{align*}
\langle \cdot,\cdot\rangle_{\wedge^k}:\quad \wedge^k(\mathbb{C}^n)\times\wedge^k(\mathbb{C}^n)\ &\longrightarrow\ \mathbb{C}\\
\langle u_1\wedge\cdots\wedge u_k,v_1\wedge\cdots\wedge v_k\rangle_{\wedge^k}\ &= 
\det\left[\begin{array}{cccc}
 \langle u_1,v_1\rangle & \langle u_1,v_2\rangle  & \cdots & \langle u_1,v_k\rangle \\
 \langle u_2,v_1\rangle & \langle u_2,v_2\rangle & \cdots & \langle u_2,v_k\rangle\\
 \vdots & \vdots & \ddots & \vdots \\
 \langle u_k,v_1\rangle & \langle u_k,v_2\rangle & \cdots & \langle u_k,v_k\rangle
\end{array}\right],
\end{align*}
where $\langle u,v\rangle=u^*v$ is the standard $l_2$ inner product on $\mathbb{C}^n$. Let $\mathcal{L}(\wedge^k(\mathbb{C}^n))$ denote the space of all linear operators from $\wedge^k(\mathbb{C}^n)$ to itself. For any matrix $A\in\mathbb{C}^{n\times n}$, we define $\wedge^kA \in \mathcal{L}(\wedge^k(\mathbb{C}^n))$ by
\[\wedge^kA (v_1\wedge v_2\wedge\cdots\wedge v_k) = Av_1\wedge Av_2\wedge\cdots\wedge Av_k\]
for any elementary product $v_1\wedge v_2\wedge\cdots\wedge v_k\in \wedge^k(\mathbb{C}^n)$, with linear extension to all other elements in $\wedge^k(\mathbb{C}^n)$. We will be using the following properties of $\wedge^k A$. 
\begin{itemize}
\item Invertibility: If $A\in \mathbb{C}^{n\times n}$ is invertible, then $(\wedge^kA)^{-1} = \wedge^kA^{-1}$.
\item Adjoint: For any $A\in \mathbb{C}^{n\times n}$, $(\wedge^kA)^*=\wedge^kA^*$, with respect to the inner product $\langle \cdot,\cdot\rangle_{\wedge^k}$. In particular, if $A\in \mathbf{H}^n$, then $\wedge^kA$ is Hermitian.
\item Power: For any $A\in\mathbf{H}^n$ and any $t\in\mathbb{C}$, $(\wedge^kA)^t = \wedge^kA^t$.
\item Positiveness: If $A\in \mathbf{H}_+^n$, then $\wedge^kA\succeq \mathbf{0}$; if $A\in \mathbf{H}_{++}^n$, then $\wedge^kA\succ\mathbf{0}$. 
\item Product: For any $A,B\in \mathbb{C}^{n\times n}$, $\wedge^k(AB) = (\wedge^kA)(\wedge^kB)$.
\item Spectrum: If $\{\lambda_i\}_{i=1}^n$ are all eigenvalues of $A\in \mathbb{C}^{n\times n}$, then $\{\lambda_{i_1}\lambda_{i_2}\cdots\lambda_{i_k}\}_{1\leq i_1<i_2<\cdots<i_k\leq n}$ are all eigenvalues of $\wedge^kA$. In particular, if $A\in \mathbf{H}_+^n$, then $\lambda_1(\wedge^kA) = \prod_{i=1}^k\lambda_i(A)$.
\end{itemize}
Using these properties, one can check that for any $A\in \mathbb{C}^{n\times n}$,
\[|\wedge^kA|=\big((\wedge^kA)^*\wedge^kA\big)^\frac{1}{2}=\wedge^k(A^*A)^\frac{1}{2}=\wedge^k|A|.\]

\subsection{Majorization} \label{subsec:Majorization}
Given a vector $a = (a_1,a_2,\dots,a_n) \in \mathbb{R}^n$, we use $a_{[j]}$ to denote the $j_{\text{th}}$ largest entry of $a$. For any two vector $a,b\in \mathbb{R}^n$, $a$ is said to be weakly majorized by $b$, denoted by $a\prec_w b$, if 
\[\sum_{j=1}^ka_{[j]}\leq \sum_{j=1}^kb_{[j]}, \quad 1\leq k\leq n;\] 
moreover, a is said to be majorized by $b$, denoted by $a\prec b$, if equality holds for $k=n$, i.e.
\[\sum_{j=1}^na_j=\sum_{j=1}^nb_j.\] 
The following two lemmas are most important for deriving inequalities from majorization relations. One may refer to \cite{ando1989majorization,marshall1979inequalities,hiai2010matrix} for proofs and more discussions on this topic. 

\begin{lemma}\label{lem:MajorBridge}
For $a,b\in \mathbb{R}^n$, if $a\prec_w b$, then there is some $c \in \mathbb{R}^n$ such that $a\leq c\prec b$.
\end{lemma}

\begin{lemma}\label{lem:MajorKey}
For $a,b\in \mathbb{R}^n$, $a\prec b$ if and only if $a=Db$ for some doubly stochastic matrix $D$, i.e. $D_{ij}\geq 0, 1\leq i,j\leq n$, $\sum_{j=1}^nD_{ij} = 1, 1\leq i\leq n$ and $\sum_{i=1}^nD_{ij} = 1, 1\leq j\leq n$.
\end{lemma}

Since any doubly stochastic matrix is a convex combination of permutation matrices, an equivalent statement of \Cref{lem:MajorKey} is that, $a\prec b$ if and only if $a$ is a convex combination of permutations of $b$, i.e.
\[a = \sum_{j=1}^m\tau_j P_jb,\]
for some $\{P_j\}_{j=1}^m\subset S_n$ and some $\{\tau_j\}_{j=1}^m\subset[0,1]$ such that $\sum_{j=1}^m\tau_j=1$.

\begin{lemma}\label{lem:MajorSymmetry}
Let $\phi:\mathbb{R}^n\rightarrow \mathbb{R}$ be convex and symmetric such that $\phi(x)=\phi(Px)$ for any $x\in \mathbb{R}^n$ and any permutation $P\in S_n$. Then for any $a,b\in \mathbb{R}^n$ such that $a\prec b$, $\phi(a)\leq \phi(b)$.  
\end{lemma}

\begin{proof} Since $a\prec b$, we have $a = \sum_{j=1}^m\tau_j P_jb$ for some permutations $\{P_j\}_{j=1}^m\subset S_n$ and some $\{\tau_j\}_{j=1}^m\subset[0,1]$ such that $\sum_{j=1}^m\tau_j=1$. Therefore
\[\phi(a) = \phi(\sum_{j=1}^m\tau_j P_jb)\leq \sum_{j=1}^m\tau_j \phi(P_jb) = \sum_{j=1}^m\tau_j \phi(b) = \phi(b). \]
\end{proof}

The next lemma shows two majorization relations between eigenvalues of matrices, which are widely used in theories of matrix norms (see e.g. \cite{marshall1979inequalities,horn2012matrix}). 

\begin{lemma}\label{lem:MatrixMajorization}
For any $A,B\in\mathbf{H}^n$, 
\begin{equation}\label{eqt:MatrixMajorization1}
\lambda(A+B)\prec \lambda(A)+\lambda(B).
\end{equation}
For any $A,B\in\mathbf{H}_+^n$, 
\begin{equation}\label{eqt:MatrixMajorization2}
\log \lambda(|AB|)\prec \log (\lambda(A)\lambda(B)).
\end{equation}

\begin{proof} For any Hermitian matrix $A\in\mathbf{H}^n$ and any $1\leq k\leq n$, we have that 
\[\sum_{j=1}^k\lambda_j(A) = \max_{\begin{subarray}{c} U\in \mathbb{C}^{n\times k}\\ U^*U = I_k \end{subarray}}\trace[U^*AU].\]
Therefore, for any $A,B\in\mathbf{H}^n$, we have
\begin{align*}
\sum_{j=1}^k\lambda_j(A+B) =&\ \max_{\begin{subarray}{c} U\in \mathbb{C}^{n\times k}\\ U^*U = I_k \end{subarray}}\trace[U^*(A+B)U]\\
\leq&\ \max_{\begin{subarray}{c} U\in \mathbb{C}^{n\times k}\\ U^*U = I_k \end{subarray}}\trace[U^*AU] + \max_{\begin{subarray}{c} U\in \mathbb{C}^{n\times k}\\ U^*U = I_k \end{subarray}}\trace[U^*BU]\\
=&\ \sum_{j=1}^k\lambda_j(A) + \sum_{j=1}^k\lambda_j(B).
\end{align*}
And obviously we have 
\[\sum_{j=1}^n\lambda_j(A+B)=\trace[A+B]=\trace[A]+\trace[B]= \sum_{j=1}^n\lambda_j(A) + \sum_{j=1}^n\lambda_j(B).\]
Therefore $\lambda(A+B)\prec \lambda(A)+\lambda(B)$.

For any $A,B\in\mathbf{H}_+^n$, we have
\[\lambda_1(|AB|) = \|BA^2B\|^\frac{1}{2}\leq \|A\|\|B\| = \lambda_1(A)\lambda_1(B). \]
Substituting $A,B$ with $\wedge^kA,\wedge^kB$ respective, we get 
\[\lambda_1(|(\wedge^kA)(\wedge^kB)|)\leq \lambda_1(\wedge^kA)\lambda_1(\wedge^kB). \]
Since $\lambda_1(\wedge^kA) = \prod_{j=1}^k\lambda_j(A)$, $\lambda_1(\wedge^kB) = \prod_{j=1}^k\lambda_j(B)$, and 
\[\lambda_1(|(\wedge^kA)(\wedge^kB)|) = \lambda_1(\wedge^k|AB|) = \prod_{j=1}^k\lambda_j(|AB|),\]
we immediately have that, for $1\leq k\leq n$.
\[\prod_{j=1}^k\lambda_j(|AB|)\leq \prod_{j=1}^k\lambda_j(A)\lambda_j(B).\]
Obviously, we have
\[\prod_{j=1}^n\lambda_j(|AB|) = \det[|AB|] = \det[A]\det[B] = \prod_{j=1}^n\lambda_j(A)\lambda_j(B).\]
Therefore $\log \lambda(|AB|)\prec \log (\lambda(A)\lambda(B))$.
\end{proof}

\end{lemma}

\subsection{Properties of symmetric forms}
\label{subsec:SymFormProperty}

\begin{lemma}\label{lem:PhiExp}
A symmetric form $\phi$ is H\"older if any only if $\phi\circ\exp$ is convex on $\mathbb{R}^n$.
\end{lemma}

\begin{proof}
If $\phi$ is H\"older, then for any $x,y\in \mathbb{R}^n$ and any $\tau\in(0,1)$, we have
\[\phi[\exp(\tau x+(1-\tau)y)]\leq \phi(\exp x)^\tau\phi(\exp y )^{1-\tau}\leq \tau\phi(\exp x )+(1-\tau)\phi(\exp y).\]
Thus $\phi\circ\exp$ is convex on $\mathbb{R}^n$. Conversely, if $\phi\circ\exp$ is convex, then for any $x,y\in \mathbb{R}^n_+,t\in \mathbb{R}_+$ and any $\tau\in(0,1]$, we have
\begin{align*}
\phi(x y) =&\ \phi[\exp\big(\tau \log( (t x)^\frac{1}{\tau})+(1-\tau)\log ((\frac{y}{t})^\frac{1}{1-\tau})\big)]\\
\leq&\ \tau \phi[\exp\log( (t x)^\frac{1}{\tau})] + (1-\tau)\phi[\exp\log ((\frac{y}{t})^\frac{1}{1-\tau})]\\
=&\ \tau t^\frac{1}{\tau} \phi(x^\frac{1}{\tau}) + \frac{(1-\tau)}{t^\frac{1}{1-\tau}}\phi(y^\frac{1}{1-\tau}).
\end{align*}
Minimizing the last line above with $t = \Big(\frac{\phi(y^\frac{1}{1-\tau})}{\phi(x^\frac{1}{\tau})}\Big)^{\tau(1-\tau)}$ gives 
\[\phi(x y)\leq \phi(x^\frac{1}{\tau})^\tau \phi(y^\frac{1}{1-\tau})^{1-\tau}.\]
The extreme cases when $\tau = 0,1$ can be obtained by continuity. Thus $\phi$ is H\"older.
\end{proof}

\begin{lemma}\label{lem:MatrixHolder}
If a symmetric form $\phi$ is H\"older, then its extension to $\mathbf{H}_+^n$ is also H\"older, i.e.
\[\phi(|AB|)\leq \phi(A^p)^\frac{1}{p}\phi(B^q)^\frac{1}{q},\]
for any $A,B\in \mathbf{H}_+^n$ and any $p,q\in[1,+\infty],\frac{1}{p}+\frac{1}{q}=1$.
\end{lemma}
\begin{proof}
By \Cref{lem:MatrixMajorization}, we have $\log \lambda(|AB|)\prec \log (\lambda(A)\lambda(B))$, i.e. $\log \lambda(|AB|)$ is a convex combination of permutations of $\log (\lambda(A)\lambda(B))$. Since $\phi$ is H\"older, $\phi\circ\exp$ is symmetric and convex. Then by \Cref{lem:MajorSymmetry} we have
\[\phi[\lambda(|AB|)]= \phi[\exp\log \lambda(|AB|)] \leq \phi[\exp\log (\lambda(A)\lambda(B))] = \phi[\lambda(A)\lambda(B)].\]
Therefore
\[\phi(|AB|) = \phi[\lambda(|AB|)]\leq \phi[\lambda(A)\lambda(B)]\leq \phi[\lambda(A)^p]^\frac{1}{p}\phi[\lambda(B)^q]^\frac{1}{q} = \phi(A^p)^\frac{1}{p}\phi(B^q)^\frac{1}{q}.\]
\end{proof}

\begin{lemma}[\textbf{Araki-Lieb-Thirring Type Inequality}]
\label{lem:ALT}
If a symmetric form $\phi$ is H\"older, then for any $A,B\in\mathbf{H}_+^n$ and any $s\geq t>0$, 
\begin{equation}\label{eqt:ALT}
\phi\big[(B^\frac{t}{2}A^tB^\frac{t}{2})^\frac{1}{t}\big]\leq \phi\big[(B^\frac{s}{2}A^sB^\frac{s}{2})^\frac{1}{s}\big].
\end{equation}
\end{lemma}

\begin{proof} This proof is due to Araki \cite{araki1990inequality}. For any $p\geq 1$, from an inequality of Heinz that 
\[\|B^\frac{p}{2}A^pB^\frac{p}{2}\|^\frac{1}{p} =\|A^\frac{p}{2}B^\frac{p}{2}\|^\frac{2}{p} \geq \|A^\frac{1}{2}B^\frac{1}{2}\|^{2} = \|B^\frac{1}{2}AB^\frac{1}{2}\|, \]
we obtain that $\lambda_1((B^\frac{p}{2}A^pB^\frac{p}{2})^\frac{1}{p})\geq \lambda_1(B^\frac{1}{2}AB^\frac{1}{2})$. Substituting $A,B$ by $\wedge ^{k}A,\wedge ^{k}B$ respectively, we further obtain that 
\begin{align*}
\prod_{j=1}^k\lambda_j((B^\frac{p}{2}A^pB^\frac{p}{2})^\frac{1}{p}) =&\ \lambda_1\Big(\big((\wedge ^{k}B)^\frac{p}{2}(\wedge ^{k}A)^p(\wedge ^{k}B)^\frac{p}{2}\big)^\frac{1}{p}\Big)\\
\geq &\ \lambda_1\Big((\wedge ^{k}B)^\frac{1}{2}(\wedge ^{k}A)(\wedge ^{k}B)^\frac{1}{2}\Big) = \prod_{j=1}^k\lambda_j(B^\frac{1}{2}AB^\frac{1}{2}).
\end{align*}
Therefore $\log \lambda((B^\frac{p}{2}A^pB^\frac{p}{2})^\frac{1}{p})\succ \log \lambda(B^\frac{1}{2}AB^\frac{1}{2})$, i.e. $\log \lambda(B^\frac{1}{2}AB^\frac{1}{2})$ is a convex combination of permutations of $\log \lambda((B^\frac{p}{2}A^pB^\frac{p}{2})^\frac{1}{p})$. Since $\phi$ is H\"older, $\phi\circ\exp$ is convex on $\mathbb{R}^n$, by \Cref{lem:MajorSymmetry} we have
\[\phi(B^\frac{1}{2}AB^\frac{1}{2})=\phi\big[\exp\big(\log \lambda(B^\frac{1}{2}AB^\frac{1}{2})\big)\big]\leq \phi\big[\exp\big(\log \lambda((B^\frac{p}{2}A^pB^\frac{p}{2})^\frac{1}{p})\big)\big] = \phi((B^\frac{p}{2}A^pB^\frac{p}{2})^\frac{1}{p}).\]
Then substituting $A,B$ by $A^t,B^t$ and choosing $p=\frac{s}{t}$ yields inequality \eqref{eqt:ALT}.
\end{proof}

\begin{lemma}[\textbf{Golden-Thompson Type Inequality}]
\label{lem:GT} 
If a symmetric form $\phi$ is H\"older, then for any $A,B\in\mathbf{H}^n$, 
\begin{equation}\label{eqt:GT}
\phi[\exp(A+B)]\leq \phi\big[\big|\exp(A)\exp(B)\big|\big].
\end{equation}
\end{lemma}
\begin{proof} If we choose take $s=2,t\rightarrow0$ and replace $A,B$ by $\exp(A),\exp(B)$ respectively in inequality \eqref{eqt:ALT}, the right hand side becomes 
\[\phi\big[\big(\exp(B)\exp(2A)\exp(B)\big)^\frac{1}{2}\big] = \phi\big[\big|\exp(A)\exp(B)\big|\big],\]
while the left hand side becomes 
\begin{align*}\lim_{t\searrow0}\phi\big[\big(\exp(\frac{t}{2}B)\exp(tA)\exp(\frac{t}{2}B)\big)^\frac{1}{t}\big] =&\ \phi\big[\lim_{t\searrow0}\big(\exp(\frac{t}{2}B)\exp(tA)\exp(\frac{t}{2}B)\big)^\frac{1}{t}\big]\\
=&\ \phi[\exp(A+B)], 
\end{align*}
where we have used the Lie product formula that $\lim_{t\searrow0}\big(\exp(\frac{t}{2}B)\exp(tA)\exp(\frac{t}{2}B)\big)^\frac{1}{t} = \exp(A+B)$. So we obtain inequality \eqref{eqt:GT}.
\end{proof}

\begin{lemma}\label{lem:MaxtrixPhiExp}
If a symmetric form $\phi$ is H\"older, then $A\mapsto \phi(\exp(A))$ is convex on $\mathbf{H}^n$.
\end{lemma}
\begin{proof}
For any $A,B\in \mathbf{H}^n$ and any $\tau\in[0,1]$, we have
\begin{align*} 
\phi[\exp(\tau A+(1-\tau)B)]\leq&\ \phi\big[\big|\exp(\tau A)\exp((1-\tau)B)\big|\big]\\
\leq&\ \phi[\exp(A)]^\tau \phi[\exp(B)]^{1-\tau}\\
\leq&\ \tau \phi[\exp(A)] + (1-\tau)\phi[\exp(B)].
\end{align*}
The first inequality is Golden-Thompson (\Cref{lem:GT}), and the second inequality is H\"older's (\Cref{lem:ALT}).
\end{proof}

\begin{lemma}[\textbf{Concavity/Convexity Preserving}]
\label{lem:ConPreserving}
Let $\phi$ be a concave (or convex) symmetric form, and $f:\mathbb{R}_+\rightarrow \mathbb{R}_+$ be a concave (or convex) function. Then the map $A\mapsto \phi(f(A))$ is concave (or convex) on $\mathbf{H}_+^n$. In particular, $A\mapsto \phi(A)$ is concave (or convex) on $\mathbf{H}_+^n$.
\end{lemma}

\begin{proof} We only prove the concave case here; the proof for the convex case is similar. For any $A,B\in \mathbf{H}_+^n$ and any $\tau\in [0,1]$, let $C=\tau A+(1-\tau) B$. We need to show that $\phi[f(C)]\geq \tau\phi[f(A)] + (1-\tau)\phi[f(B)]$. By \Cref{lem:MajorKey} and \Cref{lem:MatrixMajorization}, we know that 
\[\lambda(C) = D\big(\tau \lambda(A)+(1-\tau)\lambda(B)\big)\]
for some doubly stochastic matrix $D$. Define $x = D\big(\tau f(\lambda(A)) + (1-\tau)f(\lambda(B))\big)\in \mathbb{R}_+^n$, so $x\prec \tau f(\lambda(A)) + (1-\tau)f(\lambda(B))$. Then since $\phi$ is concave, by \Cref{lem:MajorSymmetry} we have that
\begin{align*}
\phi(x)\geq&\ \phi[\tau f(\lambda(A)) + (1-\tau)f(\lambda(B))]\\
\geq&\ \tau\phi[f(\lambda(A))]+(1-\tau)\phi[f(\lambda(B))]\\
=&\ \tau\phi[f(A)] + (1-\tau)\phi[f(B)].
\end{align*}
On the other hand, since $f$ is concave, we have
\begin{align*}
f(\lambda_i(C)) =&\ f\Big(\sum_{j=1}^nD_{ij}\big(\tau \lambda_j(A)+(1-\tau) \lambda_j(B)\big)\Big)\\
\geq&\ \sum_{j=1}^nD_{ij}\big(\tau f(\lambda_j(A))+(1-\tau) f(\lambda_j(B))\big) = x_i,\quad 1\leq i\leq n,
\end{align*}
thus $f(\lambda(C))\geq x$. Since $\phi$ is monotone, we have $\phi[f(C)]=\phi[f(\lambda(C))]\geq \phi(x)$. So finally we have
\[\phi[f(C)]\geq \tau\phi[f(A)] + (1-\tau)\phi[f(B)].\]
\end{proof}

\subsection{Operator interpolation}
We will be using Stein's interpolation of linear operators \cite{stein1956interpolation}, which was developed from Hirschman's improvement \cite{hirschman1952convexity} of the Hadamard three-line theorem \cite{hadamard1899theoreme}. Stein's technique was recently adopted by Sutter et al. \cite{sutter2017multivariate} to establish a multivariate extension of the Golden-Thompson inequality, which covers the original Golden-Thompson inequality and its three-matrix extension by Lieb \cite{LIEB1973267}. We will follow the notations in \cite{sutter2017multivariate}. For any $\theta\in(0,1)$, we define a density $\beta_\theta(t)$ on $\mathbb{R}$ by 
\begin{equation} 
\beta_\theta(t) = \frac{\sin(\pi\theta)}{2\theta\big(\cosh(\pi t)+\cos(\pi\theta)\big)},\quad t\in\mathbb{R}.
\end{equation} 
Specially, we define 
\[\beta_0(t) = \lim_{\theta\searrow0}\beta_\theta(t) = \frac{\pi}{2(\cosh(\pi t)+1)},\quad \text{and}\quad \beta_1(t) = \lim_{\theta\nearrow1}\beta_\theta(t) = \delta(t).\]
$\beta_\theta(t)$ is a density since $\beta_\theta(t)\geq0,t\in\mathbb{R}$ and $\int_{-\infty}^{+\infty}\beta_\theta(t)dt=1$. We will always use $\mathcal{S}$ to denote a vertical strip on the complex plane $\mathbb{C}$:
\begin{equation}\label{eqt:S}
\mathcal{S}=\{z\in \mathbb{C}:0\leq \mathrm{Re}(z)\leq 1\}.
\end{equation}

\begin{thm}[Stein-Hirschman]\label{thm:SHInterpolation}
Let $G(z)$ be a map from $\mathcal{S}$ to bounded linear operators on a separable Hilbert space that that is holomorphic in the interior of $\mathcal{S}$ and continuous on the boundary. Let $p_0,p_1\in[1,+\infty],\theta\in[0,1]$, and define $p_\theta$ by 
\[\frac{1}{p_\theta}=\frac{1-\theta}{p_0}+\frac{\theta}{p_1}.\]
Then if $\|G(z)\|_{p_{\mathrm{Re}(z)}}$ is uniformly bounded on $\mathcal{S}$, the following inequality holds:
\begin{equation}\label{eqt:SHInterpolation}
\log\|G(\theta)\|_{p_\theta} \leq  \int_{-\infty}^{+\infty}dt\Big(\beta_{1-\theta}(t)\log\|G(it)\|_{p_0}^{1-\theta}+\beta_\theta(t)\log\|G(1+it)\|_{p_1}^\theta\Big).
\end{equation}
\end{thm}

By choosing $p_0=p_1=p_\theta$ and taking $p_\theta\rightarrow +\infty$, we obtain from inequality \eqref{eqt:SHInterpolation} that 
\[\log\lambda_1(|G(\theta)|) \leq  \int_{-\infty}^{+\infty}dt\Big((1-\theta)\beta_{1-\theta}(t)\log\lambda_1\big(|G(it)|\big)+\theta\beta_\theta(t)\log\lambda_1\big(|G(1+it)|\big)\Big).\]
It is easy to see that, if $G(z)$ satisfies the assumptions in \Cref{thm:SHInterpolation}, so is $\wedge^kG(z)$ for any $1\leq k\leq n$. Thus we may replace $G(z)$ by $\wedge^kG(z)$ in \eqref{eqt:SHInterpolation} and use the fact that 
\[\lambda_1(|\wedge^kG(z)|)=\lambda_1(\wedge^k|G(z)|)=\prod_{j=1}^k\lambda_j(|G(z)|)\] 
to obtain a majorization relation
\begin{equation}\label{eqt:InterMajor}
\sum_{j=1}^k\log\lambda_j(|G(\theta)|) \leq  \sum_{j=1}^k\int_{-\infty}^{+\infty}dt\Big((1-\theta)\beta_{1-\theta}(t)\log\lambda_j\big(|G(it)|\big)+\theta\beta_\theta(t)\log\lambda_j\big(|G(1+it)|\big)\Big).
\end{equation}
The above arguments follow from the work of Hiai et al. \cite{hiai2017generalized}, in which they proved generalized log-majorization theorems in form of \eqref{eqt:InterMajor} for general distributions instead of this particular $\beta_\theta$. This majorization relation grants us the following lemma. 

\begin{lemma}\label{lem:KeyLemma}
Let $G(z):\mathcal{S}\rightarrow\mathbb{C}^{n\times n}$ be a map satisfying the assumptions in \Cref{thm:SHInterpolation}. Then for any symmetric form $\phi$ that is H\"older, and any $\theta\in[0,1], p_\theta,p_0,p_1\in(0,+\infty)$ satisfying $\frac{1}{p_\theta}=\frac{1-\theta}{p_0}+\frac{\theta}{p_1}$, the following inequality holds:
\begin{equation}\label{eqt:PhiInterpolation}
\phi\big(|G(\theta)|^{p_\theta}\big) \leq  \int_{-\infty}^{+\infty}dt\Big(\frac{(1-\theta)p_\theta}{p_0}\beta_{1-\theta}(t)\phi\big(|G(it)|^{p_0}\big)+\frac{\theta p_\theta}{p_1}\beta_\theta(t)\phi\big(|G(1+it)|^{p_1}\big)\Big).
\end{equation}
\end{lemma}
\begin{proof}
Define $x = \log \lambda(|G(\theta)|^{p_\theta})\in\mathbb{R}^n$ and 
\[y = \int_{-\infty}^{+\infty}dt\Big(\frac{(1-\theta)p_\theta}{p_0}\beta_{1-\theta}(t)\log\lambda\big(|G(it)|^{p_0}\big)+\frac{\theta p_\theta}{p_1}\beta_\theta(t)\log\lambda\big(|G(1+it)|^{p_1}\big)\Big)\in\mathbb{R}^n.\]
From inequality \eqref{eqt:InterMajor} we have that $x\prec_w y$. By \Cref{lem:MajorBridge}, there is some $v\in\mathbb{R}^n$ such that $x\leq v\prec y$. Since $\phi$ and $\exp$ are both monotone increasing, we have 
\[\phi\big(|G(\theta)|^{p_\theta}\big) = \phi\big[\exp\log\lambda(|G(\theta)|^{p_\theta})\big] = \phi(\exp(x))\leq \phi(\exp(v)).\]
And since $\phi$ is H\"older, $\phi\circ\exp$ is convex, we have 
\begin{align*}
\phi(\exp(v)) \leq&\ \phi(\exp(y))\\
=&\ \phi\Big[\exp\int_{-\infty}^{+\infty}dt\Big(\frac{(1-\theta)p_\theta}{p_0}\beta_{1-\theta}(t)\log\lambda\big(|G(it)|^{p_0}\big)+\frac{\theta p_\theta}{p_1}\beta_\theta(t)\log\lambda\big(|G(1+it)|^{p_1}\big)\Big)\Big]\\
\leq&\ \int_{-\infty}^{+\infty}dt\Big(\frac{(1-\theta)p_\theta}{p_0}\beta_{1-\theta}(t)\phi\big[\exp\log\lambda\big(|G(it)|^{p_0}\big)\big]+\frac{\theta p_\theta}{p_1}\beta_\theta(t)\phi\big[\exp\log\lambda\big(|G(1+it)|^{p_1}\big)\big]\Big)\\
=&\ \int_{-\infty}^{+\infty}dt\Big(\frac{(1-\theta)p_\theta}{p_0}\beta_{1-\theta}(t)\phi\big(|G(it)|^{p_0}\big)+\frac{\theta p_\theta}{p_1}\beta_\theta(t)\phi\big(|G(1+it)|^{p_1}\big)\Big).
\end{align*}
The second inequality above is Jensen's and due to 
\[\int_{-\infty}^{+\infty}dt\Big(\frac{(1-\theta)p_\theta}{p_0}\beta_{1-\theta}(t)+\frac{\theta p_\theta}{p_1}\beta_\theta(t)\Big) = \frac{(1-\theta)p_\theta}{p_0} + \frac{\theta p_\theta}{p_1} = 1.\]
\end{proof}

The following derivation is due to Sutter et al. \cite{sutter2017multivariate}. If we choose 
\[G(z) = \prod_{j=1}^mA_j^z = A_1^zA_2^z\cdots A_m^z\]
for some $\{A_j\}_{j=1}^m\subset\mathbf{H}_+^n$, and take $p_0\rightarrow +\infty,p_\theta = \frac{p}{\theta},p_1=p$ for some $p\in(0,+\infty)$, we obtain from \Cref{lem:KeyLemma} that 
\begin{equation}\label{eqt:PhiSBT}
\log\phi\Big(\Big|\prod_{j=1}^mA_j^\theta\Big|^\frac{p}{\theta}\Big) \leq \int_{-\infty}^{+\infty}dt\beta_\theta(t)\log\phi\Big(\Big|\prod_{j=1}^mA_j^{1+it}\Big|^p\Big).
\end{equation}
If we further replace replace $A_j$ in inequality \eqref{eqt:PhiSBT} by $\exp(A_j)$, and take $\theta\rightarrow0$, the right hand side of \eqref{eqt:PhiSBT} converges  
\[\int_{-\infty}^{+\infty}dt\beta_0(t)\log\phi\Big(\Big|\prod_{j=1}^m\exp\big((1+it)A_j\big)\Big|^p\Big),\]
since each $\|\exp((1+it)A_j)\|=\|\exp(A_j)\exp(itA_j)\|=\|\exp(A_j)\|$ is uniformly bounded for all $t\in\mathbb{R}$. Moreover, by a multivariate Lie product formula (see e.g. \cite{sutter2017multivariate})
\[\lim_{\theta\searrow0}\Big(\exp(\theta X^{(1)})\exp(\theta X^{(2)})\cdots\exp(\theta X^{(m)})\Big)^\frac{1}{\theta}=\exp\big(\sum_{i=1}^mX_j\big),\]
the left hand side of \eqref{eqt:PhiSBT} becomes 
\begin{align*}
\lim_{\theta\searrow0}\log\phi\Big(\Big|\prod_{j=1}^m\exp\big(\theta A_j\big)\Big|^\frac{p}{\theta}\Big) =&\ \lim_{\theta\searrow0}\log\phi\Big(\Big(\prod_{j=1}^m\exp\big(\theta A_{m-j+1}\big)\prod_{j=1}^m\exp\big(\theta A_j\big)\Big)^\frac{p}{2\theta}\Big)\\
=&\ \log\phi\Big(\Big(\exp\big(\sum_{j=1}^m2A_j\big)\Big)^\frac{p}{2}\Big)\\
=&\ \log\phi\Big(\Big(\exp\big(\sum_{j=1}^mA_j\big)\Big)^p\Big).
\end{align*}
We therefore obtain the following. 
\begin{corollary}\label{cor:PhiMultiGT}
If a symmetric form $\phi$ is H\"older, then for any $A_1,A_2,\dots,A_m\in \mathbf{H}^n$, the following inequality holds:
\begin{equation}\label{eqt:PhiMultiGT}
\log\phi\Big(\Big(\exp\big(\sum_{j=1}^mA_j\big)\Big)^p\Big)\leq \int_{-\infty}^{+\infty}dt\beta_0(t)\log\phi\Big(\Big|\prod_{j=1}^m\exp\big((1+it)A_j\big)\Big|^p\Big).
\end{equation}
\end{corollary}

\Cref{cor:PhiMultiGT} can be seen as a multivariate extension of Golden-Thompson inequality for H\"older symmetric forms. Sutter et al. proved inequality \eqref{eqt:PhiMultiGT} for Schatten p-norms $\phi = \|\cdot\|_p$; Hiai et al. improved this result to more general $\phi=\|f(\cdot)\|$ for any unitary invariant matrix norm $\|\cdot\|$ and any continuous function $f$ such that $\log \circ f\circ\exp$ is convex. We further extend their results to any symmetric form $\phi$ that is H\"older. 

If we choose $m = 2,p=2$ in \Cref{cor:PhiMultiGT} and replace $A_j$ by $\frac{1}{2}A_j$, the right hand side of inequality \eqref{eqt:PhiMultiGT} is independent of $t$ since $\phi$ is unitary invariant. We then recover the Golden-Thompson inequality 
\[\phi\big(\exp(A_1+A_2\big)\leq \phi\big(\exp(A_1)\exp(A_2)\big)\]
that we have obtained in \Cref{lem:GT}. If we choose $m = 3,p=2$ in \Cref{cor:PhiMultiGT} and again replace $A_j$ by $\frac{1}{2}A_j$, we have 
\begin{align*}
&\ \log\phi\big(\exp(A_1+A_2+A_3)\big)\\
\leq&\ \int_{-\infty}^{+\infty}dt\beta_0(t)\log\phi\Big( \exp(A_1)\exp\big(\frac{1+it}{2}A_2\big)\exp(A_3)\exp\big(\frac{1-it}{2}A_2\big) \Big)\\
\leq&\ \log\phi\left( \int_{-\infty}^{+\infty}dt\beta_0(t)\exp(A_1)\exp\big(\frac{1+it}{2}A_2\big)\exp(A_3)\exp\big(\frac{1-it}{2}A_2\big) \right)
\end{align*}
The second inequality above is due to concavity of logarithm and $\phi$. If we define 
\[\mathcal{T}_A[B]=\int_0^{+\infty}dt(A+tI_n)^{-1}B(A+tI_n)^{-1},\quad A,B\in\mathbf{H}_{++}^n,\] 
and use Lemma 3.4 in \cite{sutter2017multivariate} that 
\[\int_0^{+\infty}dt(A^{-1}+tI_n)^{-1}B(A^{-1}+tI_n)^{-1} = \int_{-\infty}^{+\infty}dt\beta_0(t)A^\frac{1+it}{2}BA^\frac{1-it}{2}, \quad A,B\in\mathbf{H}_{++}^n,\]
we then further obtain 
\[\phi\big(\exp(A_1+A_2+A_3)\big) \leq \phi\big(\exp(A_1)\mathcal{T}_{\exp(-A_2)}[\exp(A_3)]\big).\]
This can be seen as a generalization of Lieb's \cite{LIEB1973267} three-matrix extension of the Golden-Thompson inequality that $\trace[\exp(A+B+C)]\leq \trace[\exp(A)\mathcal{T}_{\exp(-B)}[\exp(C)]].$

\section{Proof of main theorems}
\label{sec:Proofs}
The proofs of \Cref{lem:GeneralEpstein},\Cref{thm:GeneralLiebConcavity} and \Cref{thm:GeneralLieb} follow from Huang's work in \cite{huang2019generalizing}, where he applied similar strategies to specific symmetric forms $\phi = \trace_k[\cdot]^\frac{1}{k}$. The key of applying \Cref{lem:KeyLemma} is to choose some proper holomorphic function $G(z)$ and then interpolating on some power in $[0,1]$. In particular, we will interpolation on $s$ to prove \Cref{lem:GeneralEpstein}, and then on $p$ to prove \Cref{thm:GeneralLiebConcavity}.

\begin{proof}[\rm\textbf{Proof of \Cref{lem:GeneralEpstein}}]
We need to show that, for any $A,B\in \mathbf{H}_+^n$ and any $\tau\in[0,1]$, 
\[\tau\phi\big((K^*A^{rs}K)^{\frac{1}{s}}\big)  +(1-\tau)\phi\big((K^*B^{rs}K)^{\frac{1}{s}}\big)  \leq \phi\big((K^*C^{rs}K)^{\frac{1}{s}}\big)  ,\]
where $C=\tau A+(1-\tau)B$. We may assume that $A,B\in \mathbf{H}_{++}^n$ and $K$ is invertible. Once this is done, the general result for $A,B\in \mathbf{H}_+^n$ and $K\in\mathbb{C}^{n\times n}$ can be obtained by continuity. Let $M = C^\frac{rs}{2}K$, and let $M=Q|M|$ be the polar decomposition of $M$ for some unitary matrix $Q$. Since $C,K$ are both invertible, $|M|\in\mathbf{H}_{++}^n$. We then define two functions from $\mathcal{S}$ to $\mathbb{C}^{n\times n}$:
\[G_A(z) = A^\frac{rz}{2}C^{-\frac{rz}{2}}Q|M|^\frac{z}{s},\quad G_B(z) = B^\frac{rz}{2}C^{-\frac{rz}{2}}Q|M|^\frac{z}{s},\quad z\in\mathcal{S},\]
where $\mathcal{S}$ is given by \eqref{eqt:S}. In what follows we will use $X$ for $A$ or $B$. We then have
\begin{align*}
\phi\big((K^*X^{rs}K)^{\frac{1}{s}}\big) =&\  \phi\big((M^*C^{-\frac{rs}{2}}X^{rs}C^{-\frac{rs}{2}}M)^\frac{1}{s}\big)\\
=&\ \phi\big((|M|Q^*C^{-\frac{rs}{2}}X^\frac{rs}{2}X^\frac{rs}{2}C^{-\frac{rs}{2}}Q|M|)^\frac{1}{s}\big)\\
=&\ \phi\big(|G_X(s)|^\frac{2}{s}\big).
\end{align*}
Since $X,C,M$ are now fixed matrices in $\mathbf{H}_{++}^n$, $G_X(z)$ is apparently holomorphic in the interior of $\mathcal{S}$ and continuous on the boundary. Also, it is easy to check that $\|G_X(z)\|$ is uniformly bounded on $\mathcal{S}$, since $\mathrm{Re}(z)\in[0,1]$. Therefore we can use \Cref{lem:KeyLemma} with $\theta = s,p_\theta = \frac{2}{s}$ to obtain 
\begin{align*}
&\ \phi(|G_X(s)|^\frac{2}{s})\leq \int_{-\infty}^{+\infty}dt\Big(\frac{2(1-s)}{sp_0}\beta_{1-s}(t)\phi\big(|G_X(it)|^{p_0}\big)+\frac{2}{p_1}\beta_s(t)\phi\big(|G_X(1+it)|^{p_1}\big)\Big).
\end{align*}
We still need to choose some $p_0,p_1$ satisfying $\frac{1-s}{p_0}+\frac{s}{p_1}=\frac{1}{p_s}=\frac{s}{2}$ to proceed. Note that $G_X(it) = X^\frac{irt}{2}C^{-\frac{irt}{2}}Q|M|^\frac{it}{s}$ is now a unitary matrix for any $t\in\mathbb{R}$ since $X,C,|M|\in\mathbf{H}_{++}^n$, and thus $|G_X(it)|^{p_0} = I_n$ for all $p_0$. Therefore we can take $p_0\rightarrow +\infty,p_1=2$ to obtain 
\[\phi(|G_X(s)|^\frac{2}{s})\leq \int_{-\infty}^{+\infty}dt\beta_s(t)\phi\big(|G_X(1+it)|^2\big),\]
given that $\phi$ is H\"older. Moreover, for each $t\in\mathbb{R}$, we have 
\begin{align*}
&\ \phi\big(|G_X(1+it)|^2\big)\\
=&\ \phi\big(G_X(1+it)^*G_X(1+it)\big) \\
= &\ \phi\big(|M|^\frac{(1-it)}{s}Q^*C^{-\frac{r(1-it)}{2}}X^rC^{-\frac{r(1+it)}{2}}Q|M|^\frac{(1+it)}{s}\big)\\
=&\ \phi\big(|M|^\frac{1}{s}Q^*C^{-\frac{r(1-it)}{2}}X^rC^{-\frac{r(1+it)}{2}}Q|M|^\frac{1}{s}\big)
\end{align*}
where we have used that $\phi$ is unitary invariant. Therefore, by substituting $X=A,B$, we obtain that
\begin{align*}
&\ \tau\phi\big(|G_A(1+it)|^2\big) + (1-\tau)\phi\big(|G_B(1+it)|^2\big) \\
=&\ \tau\phi\big(|M|^\frac{1}{s}Q^*C^{-\frac{r(1-it)}{2}}A^rC^{-\frac{r(1+it)}{2}}Q|M|^\frac{1}{s}\big)\\
&\ + (1-\tau)\phi\big(|M|^\frac{1}{s}Q^*C^{-\frac{r(1-it)}{2}}B^rC^{-\frac{r(1+it)}{2}}Q|M|^\frac{1}{s}\big)\\
\leq &\ \phi\big(|M|^\frac{1}{s}Q^*C^{-\frac{r(1-it)}{2}}(\tau A^r+(1-\tau) B^r)C^{-\frac{r(1+it)}{2}}Q|M|^\frac{1}{s}\big)\\
\leq&\ \phi\big(|M|^\frac{1}{s}Q^*C^{-\frac{r(1-it)}{2}}C^rC^{-\frac{r(1+it)}{2}}Q|M|^\frac{1}{s}\big)\\
=&\ \phi\big(|M|^\frac{2}{s}\big)\\
=&\ \phi\big((M^*M)^\frac{1}{s}\big).
\end{align*}
The first inequality above is due to the concavity of $\phi$ on $\mathbf{H}_+^n$ by \Cref{lem:ConPreserving}, the second inequality is due to (i) that $\phi$ is monotone increasing on $\mathbf{H}_+^n$ and (ii) that $X\mapsto X^r$ is operator concave on $\mathbf{H}_+^n$ for $r\in(0,1]$. Finally, since $\phi\big((M^*M)^\frac{1}{s}\big)$ is independent of $t$, and $\beta_s(t)$ is a density on $\mathbb{R}$, we have that 
\begin{align*}
&\ \tau\phi\big((K^*A^{rs}K)^{\frac{1}{s}}\big)  +(1-\tau)\phi\big((K^*B^{rs}K)^{\frac{1}{s}}\big) \\
=&\ \tau\phi\big(|G_A(s)|^\frac{2}{s}\big) + (1-\tau)\phi\big(|G_B(s)|^\frac{2}{s}\big) \\ 
\leq&\ \phi\big((M^*M)^\frac{1}{s}\big)\\
=&\ \phi\big((K^*C^{rs}K)^{\frac{1}{s}}\big).
\end{align*}
So we have proved the concavity of \eqref{eqt:function1} on $\mathbf{H}_+^n$.
\end{proof}

\begin{proof}[\rm\textbf{Proof of \Cref{thm:GeneralLiebConcavity}}]
To prove the concavity of \eqref{eqt:function2}, we need to show that, for any $A_1,B_1\in \mathbf{H}_+^n,A_2,B_2\in \mathbf{H}_+^m$ and any $\tau\in[0,1]$, 
\[\tau\phi\big((A_2^\frac{qs}{2}K^*A_1^{ps}KA_2^\frac{qs}{2})^{\frac{1}{s}}\big) +(1-\tau)\phi\big((B_2^\frac{qs}{2}K^*B_1^{ps}KB_2^\frac{qs}{2})^{\frac{1}{s}}\big) \leq \phi\big((C_2^\frac{qs}{2}K^*C_1^{ps}KC_2^\frac{qs}{2})^{\frac{1}{s}}\big) ,\]
where $C_j=\tau A_j+(1-\tau)B_j,j=1,2$. Again, we may assume that $A_1,B_1\in \mathbf{H}_{++}^n,A_2,B_2\in \mathbf{H}_{++}^m$. Once this is done, the general result for $A_1,B_1\in \mathbf{H}_+^n,A_2,B_2\in \mathbf{H}_+^m$ can be obtained by continuity. Let $M = C_1^\frac{ps}{2}KC_2^\frac{qs}{2}$, and define two functions from $\mathcal{S}$ to $\mathbb{C}^{n\times n}$:
\[G_A(z) = A_1^\frac{rsz}{2}C_1^{-\frac{rsz}{2}}MC_2^{-\frac{rs(1-z)}{2}}A_2^\frac{rs(1-z)}{2},\quad z\in\mathcal{S},\]
\[G_B(z) = B_1^\frac{rsz}{2}C_1^{-\frac{rsz}{2}}MC_2^{-\frac{rs(1-z)}{2}}B_2^\frac{rs(1-z)}{2},\quad z\in\mathcal{S},\]
where $\mathcal{S}$ is given by \eqref{eqt:S}, and $r=p+q\in(0,1]$. In what follows we may use $X$ for $A$ or $B$. We then have
\begin{align*}
\phi\big((X_2^\frac{qs}{2}K^*X_1^{ps}KX_2^\frac{qs}{2})^{\frac{1}{s}}\big) =&\  \phi\big((X_2^\frac{qs}{2}C_2^{-\frac{qs}{2}}M^*C_1^{-\frac{ps}{2}}X_1^\frac{ps}{2}X_1^\frac{ps}{2}C_1^{-\frac{ps}{2}}MC_2^{-\frac{qs}{2}}X_2^\frac{qs}{2})^\frac{1}{s}\big)\\
=&\ \phi\big(|G_X\big(\frac{p}{r}\big)|^\frac{2}{s}\big).
\end{align*}
Since $X_1,X_2,C,M$ are now fixed matrices in $\mathbf{H}_{++}^n$, $G_X(z)$ is apparently holomorphic in the interior of $\mathcal{S}$ and continuous on the boundary. Also, it is easy to check that $\|G_X(z)\|$ is uniformly bounded on $\mathcal{S}$, since $\mathrm{Re}(z)\in[0,1]$. Therefore we can use \Cref{lem:KeyLemma} with $\theta = \frac{p}{r}, p_0=p_1=p_\theta=\frac{2}{s}$ to obtain 
\begin{equation*}
\phi(|G_X\big(\frac{p}{r}\big)|^\frac{2}{s})\leq \int_{-\infty}^{+\infty}dt\Big(\frac{q}{r}\beta_{1-\frac{p}{r}}(t)\phi\big(|G_X(it)|^\frac{2}{s}\big)+\frac{p}{r}\beta_\frac{p}{r}(t)\phi\big(|G_X(1+it)|^\frac{2}{s}\big)\Big),
\end{equation*} 
since $\phi$ is H\"older. Moreover, for each $t\in\mathbb{R}$, we have 
\begin{align*}
&\ \phi\big(|G_X(1+it)|^\frac{2}{s}\big)\\
=&\ \phi\Big(\big(G_X(1+it)^*G_X(1+it)\big)^\frac{1}{s}\Big) \\
= &\ \phi\Big(\big(X_2^{\frac{irst}{2}}C_2^{-\frac{irst}{2}}M^*C_1^{-\frac{rs(1-it)}{2}}X_1^{rs}C_1^{-\frac{rs(1+it)}{2}}MC_2^{\frac{irst}{2}}X_2^{-\frac{irst}{2}}\big)^\frac{1}{s}\Big)\\
=&\ \phi\Big(\big(M^*C_1^{-\frac{rs(1-it)}{2}}X_1^{rs}C_1^{-\frac{rs(1+it)}{2}}M\big)^\frac{1}{s}\Big),
\end{align*}
where we have used that (i) $X^{it}$ is unitary for any $X\in\mathbf{H}_{++}^n,t\in\mathbb{R}$, (ii) $f(U^*XU)=U^*f(X)U$ for any $X\in\mathbf{H}^n$, any unitary $U\in\mathbb{C}^{n\times n}$ and any function $f$, and (iii) $\phi$ is unitary invariant. Now since $r,s\in(0,1]$, we can use \Cref{lem:GeneralEpstein} to obtain
\begin{align*}
&\ \tau\phi\big(|G_A(1+it)|^\frac{2}{s}\big) + (1-\tau)\phi\big(|G_B(1+it)|^\frac{2}{s}\big) \\
=&\ \tau\phi\Big(\big(M^*C_1^{-\frac{rs(1-it)}{2}}A_1^{rs}C_1^{-\frac{rs(1+it)}{2}}M\big)^\frac{1}{s}\Big) \\
&\ + (1-\tau)\phi\Big(\big(M^*C_1^{-\frac{rs(1-it)}{2}}B_1^{rs}C_1^{-\frac{rs(1+it)}{2}}M\big)^\frac{1}{s}\Big)\\
\leq&\ \phi\Big(\big(M^*C_1^{-\frac{rs(1-it)}{2}}(\tau A_1+(1-\tau)B_1)^{rs}C_1^{-\frac{rs(1+it)}{2}}M\big)^\frac{1}{s}\Big)\\
=&\ \phi\Big(\big(M^*C_1^{-\frac{rs(1-it)}{2}}C_1^{rs}C_1^{-\frac{rs(1+it)}{2}}M\big)^\frac{1}{s}\Big)\\
=&\ \phi\big((M^*M)^\frac{1}{s}\big).
\end{align*}
Similarly, we have that for each $t\in\mathbb{R}$,
\begin{align*}
\phi\big(|G_X(it)|^\frac{2}{s}\big) =&\ \phi\Big(\big(X_2^{\frac{rs(1+it)}{2}}C_2^{-\frac{rs(1+it)}{2}}M^*MC_2^{-\frac{rs(1-it)}{2}}X_2^{\frac{rs(1-it)}{2}}\big)^\frac{1}{s}\Big)\\
=&\ \phi\Big(\big(MC_2^{-\frac{rs(1-it)}{2}}X_2^{rs}C_2^{-\frac{rs(1+it)}{2}}M^*\big)^\frac{1}{s}\Big).
\end{align*}
We have used the fact that $\phi\big(f(X^*X)\big) = \phi\big(f(XX^*)\big)$ for any $X\in\mathbb{C}^{n\times n}$ and any function $f$, since $\phi$ is only a function of eigenvalues and the spectrums of $f(X^*X)$ and $f(XX^*)$ are the same. Then again using \Cref{lem:GeneralEpstein} we obtain that 
\[\tau\phi\big(|G_A(it)|^\frac{2}{s}\big) + (1-\tau)\phi\big(|G_B(it)|^\frac{2}{s}\big)\leq \phi\big((MM^*)^\frac{1}{s}\big) = \phi\big((M^*M)^\frac{1}{s}\big).\]
Finally we have 
\begin{align*}
&\ \tau\phi\big((A_2^\frac{qs}{2}K^*A_1^{ps}KA_2^\frac{qs}{2})^{\frac{1}{s}}\big) +(1-\tau)\phi\big((B_2^\frac{qs}{2}K^*B_1^{ps}KB_2^\frac{qs}{2})^{\frac{1}{s}}\big) \\
=&\ \tau\phi\big(|G_A\big(\frac{p}{r}\big)|^\frac{2}{s}\big) + (1-\tau)\phi\big(|G_B\big(\frac{p}{r}\big)|^\frac{2}{s}\big) \\ 
\leq&\ \int_{-\infty}^{+\infty}dt\Big\{\frac{q}{r}\beta_{1-\frac{p}{r}}(t)\Big(\tau\phi\big(|G_A(it)|^\frac{2}{s}\big)+(1-\tau)\phi\big(|G_B(it)|^\frac{2}{s}\big)\Big) \\ 
&\ \qquad\qquad\ +\frac{p}{r}\beta_\frac{p}{r}(t)\Big(\tau\phi\big(|G_A(1+it)|^\frac{2}{s}\big)+(1-\tau)\phi\big(|G_B(1+it)|^\frac{2}{s}\big)\Big)\Big\} \\
\leq&\ \phi\big((M^*M)^\frac{1}{s}\big) \int_{-\infty}^{+\infty}\Big(\frac{q}{r}\beta_{1-\frac{p}{r}}(t) + \frac{p}{r}\beta_\frac{p}{r}(t)\Big)dt\\
=&\ \phi\big((C_2^\frac{qs}{2}K^*C_1^{ps}KC_2^\frac{qs}{2})^{\frac{1}{s}}\big).
\end{align*}
So we have proved the joint concavity of \eqref{eqt:function2}. 
\end{proof}

\begin{proof} [\rm\textbf{Proof of \Cref{thm:GeneralLieb}}]
We first prove the theorem for $m=1$. Let $r = p_1\in(0,1]$, and $K_N = K_N^* = \exp\big(\frac{1}{2N}H\big),N\geq 1$. Then using the Lie product formula, for any $A\in \mathbf{H}_{++}^n$ we have 
\begin{align*}
\lim_{N\rightarrow+\infty}\phi\Big(\big(K_N^*A^\frac{r}{N}K_N\big)^{N}\Big) =&\ \lim_{N\rightarrow+\infty}\phi\left(\Big(\exp\big(\frac{1}{2N}H\big)\exp\big(\frac{r}{N}\log A\big)\exp\big(\frac{1}{2N}H\big)\Big)^N\right)\\
=&\ \phi\big(\exp(H+r\log A)\big).
\end{align*}
By \Cref{thm:GeneralLiebConcavity}, for each $N\geq1$, $\phi\Big(\big(K_N^*A^\frac{r}{N}K_N\big)^N\Big)$ is concave in $A$, thus the limit function $\phi\big(\exp(H+r\log A)\big)$ is also concave in $A$. 

Now given any $\{A_j\}_{j=1}^m,\{B_j\}_{j=1}^m\subset\mathbf{H}_{++}^n$, and any $\tau\in[0,1]$, let $C_j = \tau A_j+(1-\tau)B_j,1\leq j\leq m$. By \Cref{lem:MaxtrixPhiExp}, $X\mapsto\phi(\exp(X))$ is convex on $\mathbf{H}^n$ since $\phi$ is H\"older, and thus $X\mapsto\phi(\exp(L+X))$ is also convex on $\mathbf{H}^n$ for arbitrary $L\in\mathbf{H}^n$. Define 
\[L= H +\sum_{j=1}^mp_j\log C_j,\quad r=\sum_{j=1}^mp_j\leq 1.\]
We then have that 
\begin{align*}
\phi\big(\exp(H+\sum_{j=1}^mp_j\log X_j)\big) =&\ \phi\Big(\exp\big(H+r\sum_{j=1}^m\frac{p_j}{r}(\log X_j-\log C_j)+ \sum_{j=1}^mp_j\log C_j\big)\Big)\\
=&\ \phi\Big(\exp\big(L+r\sum_{j=1}^m\frac{p_j}{r}(\log X_j-\log C_j)\big)\Big)\\
\leq&\ \sum_{j=1}^m\frac{p_j}{r}\phi\big(\exp(L+r\log X_j-r\log C_j)\big),\quad X_j=A_j,B_j.
\end{align*}
For each $j$, by the concavity of \eqref{eqt:function3} for $m=1$, we have 
\begin{align*}
&\ \tau \phi\big(\exp(L+r\log A_j-r\log C_j)\big) + (1-\tau)\phi\big(\exp(L+r\log B_j-r\log C_j)\big)\\
\leq &\ \phi\big(\exp(L+r\log(\tau A_j+(1-\tau) B_j)-r\log C_j)\big) \\
=&\ \phi\big(\exp(L)\big).
\end{align*}
Therefore we obtain that 
\begin{align*}
&\ \tau \phi\big(\exp(H+\sum_{j=1}^mp_j\log A_j)\big) + (1-\tau)\phi\big(\exp(H+\sum_{j=1}^mp_j\log B_j)\big)\\
\leq&\ \sum_{j=1}^m\frac{p_j}{r}\Big(\tau\phi\big(\exp(L+r\log A_j-r\log C_j)\big) + (1-\tau)\phi\big(\exp(L+r\log B_j-r\log C_j)\big)\Big)\\
\leq &\ \sum_{j=1}^m\frac{p_j}{r} \phi\big(\exp(L)\big)\\
=&\ \phi\big(\exp(H+\sum_{j=1}^mp_j\log C_j)\big),
\end{align*}
that is, \eqref{eqt:function3} is jointly concave on $(\mathbf{H}_{++}^n)^{\times m}$ for all $m\geq 1$.
\end{proof}

\section{Discussions}
\label{sec:Discussions}

The reason why we need to assume our symmetric forms to be H\"older is that we rely on operator interpolation to derive the key inequality in \Cref{lem:KeyLemma}. Roughly speaking, interpolation inequalities are essentially H\"older's inequalities. However, we conjecture that, \Cref{lem:GeneralEpstein}, \Cref{thm:GeneralLiebConcavity} and \Cref{thm:GeneralLieb} hold more generally for symmetric forms that are concave but not necessarily H\"older. In fact, our numerical evidences suggest that Lieb's concavity holds for 
\[\phi_k(x) = \sum_{i=1}^k x_{[n-i+1]},\quad x\in\mathbb{R}_+^n, 1\leq k\leq n,\] 
i.e. the sum of the $k$ smallest entries of $x$. The extension of $\phi_k$ to $\mathbf{H}_+^n$ stands for the sum of the $k$ smallest eigenvalues of a matrix. One can check that for $1\leq k<n$, $\phi_k$ is concave but not H\"older. We consider this special class of concave symmetric forms because if Lieb's concavity holds for every $\phi_k,1\leq k\leq n$, then it also holds for all concave symmetric forms. In fact, if $X\mapsto \phi_k(F(X))$ is concave on $\mathbf{H}_+^n$ for some function $F:\mathbf{H}_+^n\rightarrow \mathbf{H}_+^n$, then for any $A,B\in \mathbf{H}_+^n$ and any $\tau\in[0,1]$, we have
\[\sum_{i=1}^k\lambda_{n-i+1}(F(C))\geq \tau\sum_{i=1}^k\lambda_{n-i+1}(F(A))+(1-\tau)\sum_{i=1}^k\lambda_{n-i+1}(F(B)),\]
where $C=\tau A+(1-\tau)B$. Recall that $\lambda_i(X)$ denotes the $i_{\text{th}}$ largest eigenvalue of $X$. This means that 
\[-\lambda(F(C)) \prec_w - [\tau\lambda(F(A))+(1-\tau)\lambda(F(B))],\]
and thus by \Cref{lem:MajorBridge}, there exist some $v\in \mathbb{R}^n$ and some doubly stochastic matrix $D$ such that 
\[-\lambda(F(C)) \leq v = - D[\tau\lambda(F(A))+(1-\tau)\lambda(F(B))].\]
In particular, $-v\in \mathbb{R}_+^n$. Therefore for arbitrary concave symmetric form $\phi$, we have
\begin{align*}
\phi(F(C)) =&\ \phi[\lambda(F(C))]\\
\geq&\ \phi(-v) \\
=&\ \phi\big[D\big(\tau\lambda(F(A))+(1-\tau)\lambda(F(B))\big)\big]\\
\geq&\ \tau \phi[\lambda(F(A))] +(1-\tau) \phi[\lambda(F(B))]\\
=&\ \tau \phi[F(A)] +(1-\tau) \phi[F(B)].
\end{align*}
That is to say, the concavity of $X\mapsto\phi_k(F(X))$ for all $1\leq k\leq n$ will imply the concavity of $X\mapsto\phi(F(X))$ for arbitrary concave symmetric form $\phi$. But whether Lieb's concavity holds for $\phi_k$ with $1\leq k<n$ still remains unsolved.

\section*{Acknowledgment}
The research was in part supported by the NSF Grant DMS-1613861. The author would like to thank Thomas Y. Hou for his wholehearted mentoring and supporting.

\section*{References}
\bibliographystyle{elsarticle-num}
\bibliography{reference}

\end{document}